\documentclass[11pt]{amsart}

\usepackage{amssymb}
\usepackage{mathrsfs}
\usepackage{amsfonts}
\usepackage{latexsym,amsmath,amsthm,amssymb,amsfonts}
\usepackage[usenames]{color}
\usepackage{xspace,colortbl}
\usepackage{graphicx}
\usepackage{tipa}

\newcommand{\be}{\begin{equation}}
\newcommand{\ee}{\end{equation}}
\newcommand{\beq}{\begin{eqnarray}}
\newcommand{\eeq}{\end{eqnarray}}

\def\begeq{\begin{equation}}
\def\endeq{\end{equation}}

\def\odot{\setbox0=\hbox{$\bigcirc$}\relax \mathbin {\hbox
to0pt{\raise.5pt\hbox to\wd0{\hfil $\wedge$\hfil}\hss}\box0 }}

\numberwithin{equation} {section}

\numberwithin{equation}{section}
\newtheorem{theorem}{\bf Theorem}[section]

\newtheorem{definition}[theorem]{\bf Definition}
\newtheorem{lemma}[theorem]{\bf Lemma}

\newtheorem{corollary}[theorem]{\bf Corollary}

\allowdisplaybreaks

\begin{document}

\title[A Rigidity theorem
for parabolic 2-Hessian equations]
{A Rigidity theorem
for parabolic 2-Hessian equations}

\author{Yan He, Cen Pan and Ni Xiang$^\ast$}
\address{Faculty of Mathematics and Statistics, Hubei Key Laboratory of Applied Mathematics, Hubei University,  Wuhan 430062, P.R. China} \email{helenaig@hotmail.com, pancen960213@163.com, nixiang@hubu.edu.cn}

\thanks{This research was supported by funds from Hubei Provincial Department of Education Key Projects D20171004, D20181003. }

\thanks{$\ast$ the corresponding author}

\date{}
\begin{abstract}
In this paper, we consider the entire solutions to the parabolic $2$-Hessian equations of the form $-u_t\sigma_2(D^2 u)=1$ in $\mathbb{R}^n\times (-\infty,0]$. We prove some rigidity theorems
for the parabolic $2$-Hessian equations in $\mathbb{R}^n\times (-\infty,0]$ by establishing Pogorelov type estimates for $2$-convex-monotone solutions of the parabolic $2$-Hessian equations.

\end{abstract}

\maketitle {\it \small{{\bf Keywords}: Rigidity theorem,
parabolic 2-Hessian equation, 2-convex-monotone solution.}

{{\bf MSC}: Primary 35J60, Secondary
35B45.}
}

\section{Introduction}

Since Bernstein proved that an entire, two dimensional,
minimal graph must be a hyperplane, the Bernstein problem has been a core problem in the study of minimal submanifolds. Analytically speaking, an entire minimal graph in $R^{n+1}$ is given by an entire solution, $u(x_1,\cdots, x_n)\in C^2(R^n)$, of the following minimal equation:
\begin{equation*}
\sum _ {i=1} ^n D_i (\frac{D_iu}{\sqrt{1+|Du|^2}})=0.
\end{equation*}
The Bernstein problem asks whether an entire solution of the above equation is necessarily a linear function.
After that many problems on the classification of the entire solutions to partial differential equations have been extensively studied.

In this paper, we focus on some results concerning the rigidity theory for fully nonlinear equations. For the $k$-Hessian equations,
\begin{eqnarray}\label{sg_k}
\sigma_ {k} (D^2 u(x)) = 1.
\end{eqnarray}
Let $\sigma_k(\lambda)$ be the
$k$-th elementary symmetric function of $\lambda \in \mathbb{R}^n$.
Then $\sigma_ {k} (D^2 u(x))=\sigma_ {k} (\lambda[D^2 u])$, where
$\lambda[D^2 u]$ are the eigenvalues of the
Hessian matrix, $D^2 u$, of a function $u$ defined in $\mathbb{R}^n$.
Alternatively, it can be written as the sum of the $k \times k$
principal minors of $D^ 2 u$.

We introduce the class of functions
and domains to ensure the ellipticity of \eqref{sg_k}.
\begin{definition}
A function $u \in C^2(\mathbb{R}^n)$ is called $k$-convex
if $\lambda [D^2 u]=(\lambda_1[D^2 u], ..., \lambda_n[D^2 u])$
 belongs to $\Gamma_k$ for all $x \in \mathbb{R}^n$, where $\Gamma_k$ is the Garding's cone
\begin{eqnarray*}\label{cone}
\Gamma_{k}=\{\lambda \in \mathbb{R} ^n: \sigma_{j}(\lambda)>0, \forall 1\leq j \leq k\}.
\end{eqnarray*}
\end{definition}

Then we list some results concerning the rigidity theorems
for the entire solutions of \eqref{sg_k}. For $k = 1$, \eqref{sg_k} is a linear
equation.  Its entire convex solution must be a quadratic polynomial.
For $k = n$,  the Monge-Amp\`ere equation, a well-known theorem due to
J$\ddot{o}$rgens \cite{Jo} ($n=2$),  Calabi \cite{Ca} ($n = 3, 4, 5$) and Pogorelov \cite{Po1} \cite{Po2}
($n\geq 2$) asserts that that any entire strictly convex solution must be a
quadratic polynomial.
In 2003, Caffarelli and Li, \cite{CL} extended the theorem of J$\ddot{o}$rgens, Calabi and Pogorelov based on the theory of Monge-Amp\`ere equations \cite{Ca1,Ca2}. Moreover, Jian and Wang \cite{JW} obtained Bernstein type result for a certain Monge-Amp\`ere equation in the half space $R^n_+$.

For $k=2$, Chang and Yuan \cite{CY} obtained the rigidity for the entire convex solutions of the equation \eqref{sg_k} if the lower bound holds
\begin{eqnarray*}
D^2 u(x)\geq \bigg[\delta-\sqrt{\frac{2}{n(n-1)}}\bigg] I,
\end{eqnarray*}
for any $\delta>0$. Especially, $n=3$ and $k=2$ in \cite{Yu} by a different transformation and the geometric measure theory, rigidity theorem hold under semiconvexity assumption $D^2u\geq -KI$.
For general $k$, Bao, Chen, Guan and Ji \cite{BC} proved the Bernstein type theorem for strictly convex entire solutions of \eqref{sg_k},
satisfying a quadratic
growth are quadratic polynomials. Here the quadratic growth is defined as follows,
\begin{definition}
A function $u: \mathbb{R}^n\rightarrow \mathbb{R}$ satisfies the quadratic growth if there are
some positive constants $b, c$ and sufficiently large $R$, such that,
\begin{eqnarray}\label{gro}
u(x)\geq b|x|^2-c, \quad \mbox{for} \quad |x|\geq R.
\end{eqnarray}
\end{definition}
Recently, the strictly convex assumption
can be reduced to $(k+1)$-convexity by Li, Ren and Wang \cite{LRW}. Based on their result, Chen and Xiang \cite{CX} obtain the rigidity theorems for 2-convex solution if $\sigma_3(D^2 u) \geq -A$ and a quadratic growth \eqref{gro} when $k=2$. Especially, for $n=3$, the assumption $\sigma_3(D^2 u)\geq -A$ can be redundant.

 The parabolic Monge-Amp\`ere equation
\begin{equation}\label{060506}
-u_t\det D^2u =1 \ \ in \ \ R^n\times (-\infty,0].
\end{equation}
 was firstly proposed by Krylov \cite{Kr}. Equations (\ref{060506}) naturally appear in stochastic theory. This operator was relevant in the study of deformation of a surface by
Gauss-Kronecker curvature \cite{Fi}.

As far as we know, rigidity theorems for parabolic fully nonlinear equations are known very limited.
Guti$\acute{e}$rrez and Huang \cite{GH} extended Theorem
of J$\ddot{o}$rgens, Calabi, and Pogorelov to the parabolic Monge-Amp\`ere equations. Xiong and Bao obtained Bernstein type theorems for more general cases, such as $u_t=(\det D^2u)^{1/n}$ and $u_t = \log \det D^2u$. Then S. Nakamori and K. Takimoto \cite{NT} studied the bernstein type theorem for parabolic $k$ -Hessian equations when the entire solution $u$ was convex-monotone.

Here the function $u = u(x, t) : R^n \times (\infty , 0] \rightarrow R $ is said to be convex-monotone if it is convex in $x $ and non-increasing in $t$. Furthermore, The function $u = u(x, t) : R^n \times (\infty , 0] \rightarrow R $ is said to be k-convex-monotone if it is k-convex in $x $ and non-increasing in $t$.

It would be interesting to see if the rigidity theorem holds for general parabolic $k$-Hessian equations under $k$-convex-monotone solutions. We extend the results in our recent paper \cite{CX} from elliptic case to the parabolic 2-Hessian equations,
\begin{equation}\label{SQ}
-u_t \sigma_2 (D^2u(x)) = 1,\ \ \ in \ \ \bf \mathbb{R}^n\times (-\infty,0].
\end{equation}

Our main theorem is stated as follows.
\begin{theorem}\label{main}
Given any nonnegative constant $A$, any entire $2$-convex-monotone solution $u \in C^{4,2}(\mathbb{R}^n \times (-\infty,0])$ of the
equation \eqref{SQ}
satisfying $\sigma_{3}(D^2 u(x,t))\geq -A$ and $u(x,0)$ satisfies a quadratic growth \eqref{gro}. If there exist constants
 $m_1\geq m_2>0$ such that for all $(x,t)\in R^n\times(-\infty,0],$
 \begin{equation}\label{060401}
 -m_1\leq u_t(x,t) \leq -m_2.
 \end{equation}
Then $u$ has the form $u(x,t) = -mt + p(x)$ where the constant $m>0$ and $p$ is a quadratic polynomial.
\end{theorem}

\section{Pogorelov type lemma}

Let $W=(W_{ij})$ be a symmetric tensor and
$\sigma_{k}(W)=\sigma_k(\lambda[W])$, where $\lambda[W]$ denotes the eigenvalues of the
$W$. Similarly, we say $W \in \Gamma_2$ if $\lambda[W] \in \Gamma_2$, which also
means $\sigma_1(W)>0$, $\sigma_2(W)>0$. It follows from \cite{CNS}, if $W \in \Gamma_2$,
then $\sigma_{2}^{ij}=\frac{\partial \sigma_2}{\partial W_{ij}}(W)$ is positive definite.
We first recall the following important Lemma in \cite{CQ}.
\begin{lemma}\label{le1}
Suppose $W \in \Gamma_2$ is diagonal and $W_{11}\geq\cdot\cdot\cdot\geq W_{nn}$, if $\xi_{ij}$ is symmetric and
$$\sum_{i=2}^{n}\sigma_{2}^{ii}\xi_{ii}+\sigma_{2}^{11}\xi_{11}=\eta,$$ then
$$-\sum_{i\neq j}\xi_{ii}\xi_{jj}\geq \frac{n-1}{2\sigma_{2}(W)}
\frac{[2\sigma_{2}(W)\xi_{11}-W_{11}\eta]^2}{[(n-1)W_{11}^2+2(n-2)\sigma_{2}(W)]}
-\frac{\eta^2}{2\sigma_{2}(W)}.$$
\end{lemma}

For our case when $u \in C^{4,2}(\mathbb{R}^n\times (-\infty,0])$, let $W=D^2 u$, $-u_t\sigma_2(D^2u)=1$, $\xi_{ij}=u_{ij1}$. Thus, $\eta=\frac{u_{t1}}{u_t^2}$ and
 we obtain the following corollary directly.
\begin{corollary}\label{co1}
Let $u \in C^{4,2}(\mathbb{R}^n\times (-\infty,0])$ be a 2-convex-monotone solution of \eqref{SQ}, then
$$-\sum_{i\neq j}u_{ii1}u_{jj1}\geq
\frac{n-1}{2\sigma_2}\frac{[2\sigma_2u_{111}-u_{11}\eta]^2}{[(n-1)u_{11}^2+2(n-2)\sigma_2]}-\frac{\eta^2}{2\sigma_2}.$$
\end{corollary}

Next, we recall the following Lemma 3 in \cite{Guan1}. For completeness, we give the proof here.
\begin{lemma}\label{le2}
Under the same assumption as in Lemma \ref{le1}, and in addition that there exists
a positive constant
\begin{eqnarray}\label{le2-0}
a\leq \sqrt{\frac{\sigma_{2}(W)}{3(n-1)(n-2)}}
\end{eqnarray}
(if $n=2$, $a>0$ could be arbitrary), such that
\begin{eqnarray}\label{le2-1}
\sigma_{3}(W+aI)\geq 0,
\end{eqnarray}
then
\begin{eqnarray}\label{le2-2}
\frac{7}{6}\sigma_{2}(W)\geq (\sigma_{2}(W)+\frac{(n-1)(n-2)}{2}a^2)\geq \frac{5}{6}\sigma_{2}^{11}(W)W_{11},
\end{eqnarray}
provided that $W_{11}>6(n-2)a$. Furthermore, for any $j \in \{2,...,n\}$,
\begin{eqnarray}\label{le2-3}
|W_{jj}|\leq (n-1)^2a+\frac{7(n-1)\sigma_{2}(W)}{5W_{11}}.
\end{eqnarray}
\end{lemma}

Using the above lemma for the solution of the equation \eqref{SQ}, we can have
\begin{corollary}
Let $u$ be a 2-convex-monotone solution of \eqref{SQ}. Assume $D^2u$ is diagonal and
$u_{11}\geq\cdot\cdot\cdot\geq u_{nn}$, there exists a constant $A$ sufficiently large  such that
$$\sigma_3(D^2u)\geq -A,$$
and
\begin{eqnarray}\label{u11}
u_{11}\geq \frac{6(n-2)}{n}A,
\end{eqnarray}
then
\begin{eqnarray}\label{co2-1}
\sigma_2(D^2u)\geq \frac{5}{7}\sigma_{2}^{11}(D^2 u)u_{11},
\end{eqnarray}
and
\begin{eqnarray}\label{co2-2}
|u_{jj}|&\leq &(n-1)^2 \sqrt{\frac{2A}{n(n-1)\sigma_1m_1}}+\frac{7(n-1)\sigma_2}{5u_{11}}\\
&\leq& (n-1)^2 \sqrt{\frac{2A}{n(n-1)u_{11}m_1}}+\frac{7(n-1)}{5u_{11}m_2}.
\end{eqnarray}
\end{corollary}

\begin{proof}
We may pick $a=\sqrt{\frac{2A}{n(n-1)\sigma_{1}(D^2 u)m_1}}$. Since $u$ is 2-convex-monotone,
$u_{11}\leq\sigma_1$. Then clearly,
\begin{eqnarray*}
a=\sqrt{\frac{2A}{n(n-1)\sigma_{1}m_1}}\leq \sqrt{\frac{2A}{n(n-1)u_{11}m_1}}
\leq\sqrt{\frac{1}{3(n-1)(n-2)m_1}}\leq\sqrt{\frac{\sigma_2}{3(n-1)(n-2)}},
\end{eqnarray*}
in view of \eqref{u11}. Meanwhile,
\begin{eqnarray*}
\sigma_{3}(W+aI)&=&\sigma_{3}(W)+na\sigma_{2}(W)
+\frac{n(n-1)}{2}a^2\sigma_{1}(W)+\frac{n(n-1)(n-2)}{6}a^3
\\&\geq& \sigma_{3}(W)+\frac{n(n-1)}{2}a^2\sigma_{1}(W)\\&\geq&0,
\end{eqnarray*}
which guarantees the condition \eqref{le2-1} is satisfied. Lastly, if we choose $A$ sufficiently large, we have from \eqref{u11}
$$u_{11}^{\frac{3}{2}}\geq 6(n-2)\sqrt{\frac{2A}{n(n-1)m_1}},$$
which implies
$$u_{11}\geq 6(n-2)\sqrt{\frac{2A}{n(n-1)u_{11}m_1}}\geq 6(n-2)\sqrt{\frac{2A}{n(n-1)\sigma_{1}m_1}}=6(n-2)a.$$
Then, this corollary
follows from Lemma \ref{le2} directly.
\end{proof}

We introduce some notations. If $\Omega\subset R^n\times(-\infty,0]$ and $t\leq 0,$ $\Omega(t)$ is denoted by
$$\Omega(t) =\{x\in R^n| (x,t)\in \Omega\}.$$
Let $\Omega\subset R^n \times (-\infty,0]$ be a bounded set and $t_0=\inf\{t\leq0|\Omega(t)\neq\emptyset\}$. The parabolic boundary $\partial_p\Omega$
is defined by
$$\partial_p\Omega=(\overline{\Omega(t_0)}\times{t_0})\cup \bigcup_{t\leq 0}(\partial\Omega(t)\times {t}),$$
where $\overline{\Omega(t_0)}$ denotes the closure of $\Omega(t_0)$ and $\partial \Omega(t)$ denotes the boundary of $\Omega(t)$.

To prove Theorem \ref{main}, we need the following key Lemma.
\begin{lemma}\label{le3}
Let $\Omega$ be a bounded domain in $\mathbb{R}^n\times(-\infty,0]$, and $u \in C^{4,2}(\overline{\Omega}) $ a 2-convex-monotone solution to
\begin{equation}\label{SQ-}
\left\{
\begin{aligned}
&-u_t\sigma_2(D^2u)=1,  x \in \Omega; \\
&u=0, \quad x \in \partial_p\Omega.
\end{aligned}
\right.
\end{equation}
Assume $$\sigma_{3}(D^2 u)\geq -A,$$ for some positive constant $A$.
Then, for any $2$-convex-monotone solution $u$, we have the Pogorelov type estimate,
\begin{eqnarray}\label{est}
\max_{x \in \Omega}(-u)^{\alpha}|D^2 u|\leq C
\end{eqnarray}
for sufficiently large $\alpha>0$. Here $\alpha$ and
 $C$ only depend on $A$, $n$, $m_1$, $m_2$, $diam(\Omega_R(t))$ and  $|u|_{C^0(\Omega)}$.
\end{lemma}

\begin{proof}
 Since $u=0$ on $\partial _p\Omega$, we have $u\leq 0$ in $\Omega$ by the Comparison principle (see Theorem 17.1 in Page 443 of \cite{GT}).

 $\sigma_1 (D^2u)>0$, we obtain $$|D^2 u|\leq (n-1)\max_{\xi\in\mathbb{S}^{n-1}}u_{\xi\xi}.$$
So we need only to estimate
\begin{eqnarray*}
\max_{(x,t, \xi)\in \Omega\times \mathbb{S}^{n-1}}(-u)^{\alpha}u_{\xi\xi}\leq C.
\end{eqnarray*}
Now we consider the function for $(x,t) \in \overline{\Omega}$, $\xi \in \mathbb{S}^{n-1}$,
\begin{eqnarray*}
\widetilde{P}(x,t, \xi)=\alpha\log (-u)+\log \max \{u_{\xi\xi}, 1\}+\frac{1}{2}|x|^2,
\end{eqnarray*}
where $\alpha$ is a constant to be determined later.
By $u=0$ on $\partial_p \Omega$, the maximum of $\widetilde{P}$ is attained
in some interior point $(x_0,t_0) \in \overline{\Omega}/\partial_p\Omega$ and some $\xi(x_0) \in \mathbb{S}^{n-1}$.
Choose smooth orthonormal local frames $e_1, \ldots, e_n$
about $x_0$ such that $\xi(x_0)=e_1$ and $\{u_{ij} (x_0,t_0)\}$ is diagonal. Set
$$u_{11}(x_0,t_0)\geq u_{22}(x_0)\geq...\geq u_{nn}(x_0).$$
We may also assume that $u_{11}(x_0)\geq 1$ is sufficiently large.
Then we consider the function
\begin{eqnarray*}
P(x,t)=\alpha\log (-u)+\log u_{11}+\frac{1}{2}|x|^2.
\end{eqnarray*}
Note that $(x_0,t_0)$ is also a maximum point of $P$. We want to estimate $P(x_0,t_0)$.

At the maximum point $(x_0,t_0)$,
\begin{eqnarray}\label{1-diff}
0=P_i=\frac{\alpha u_i}{u}+\frac{u_{11i}}{u_{11}}+ x_i.
\end{eqnarray}
\begin{eqnarray}\label{060501}
0\leq P_t=\frac{\alpha u_t}{u}+\frac{u_{11t}}{u_{11}}.
\end{eqnarray}
Noticing
\begin{eqnarray}\label{060505}
0\leq P_{ij}=\frac{\alpha u_{ij}}{u}-\frac{\alpha u_i u_j}{u^2}
+\frac{u_{11ij}}{u_{11}}-\frac{u_{11i}u_{11j}}{u_{11}^2}+\delta_{ij},
\end{eqnarray}
 Differential equation \eqref{SQ} in $k$-th variable,
\begin{equation}\label{060502}
(-u_t)\sigma_{2}^{ij}u_{ijk}-\sigma_2u_{tk}=0.
\end{equation}
Taking differentiating once more of the equation \eqref{SQ},
\begin{eqnarray*}
(-u_{tl})\sigma^{ij}_2 u_{ijk}-u_t\sigma^{ij,mn}_2u_{ijk}u_{mnl}-u_t\sigma^{ij}_2u_{ijkl}-
\sigma_2^{ij}u_{ijl}u_{tk}-\sigma_2 u_{tkl}=0.
\end{eqnarray*}
By \eqref{060502}, we have
\begin{equation}\label{060503}
-u_t\sigma^{ij,mn}_2u_{ijk}u_{mnl}-u_t\sigma^{ij}_2u_{ijkl}-
2(u_t)^{-2}u_{tl}u_{tk}+(u_t)^{-1}u_{tkl}=0.
\end{equation}
Especially,
\begin{equation}\label{060602}
\frac{u_{11t}}{u_tu_{11}}-\frac{u_t}{u_{11}}\sigma^{ij}_2u_{ij11}=\frac{u_t}{u_{11}}\sigma^{ij,mn}_2u_{ij1}u_{mn1}+\frac{2u_{t1}^2}{u_{11}u_t^2}.
\end{equation}

Let $L$ be the linearized operator of \eqref{SQ} at $(x_0,t_0)$. Then we can write
\begin{equation}\label{060504}
L= \frac{1}{u_t(x_0,t_0)}D_t + (-u_t(x_0,t_0)) \sigma^{ij}_2(x_0,t_0)D_{ij}.
\end{equation}
By \eqref{1-diff}, \eqref{060501} and \eqref{060505}, we have
\begin{eqnarray*}
0&\geq & L(P)(x_0,t_0)\\
 &= & \frac{1}{u_t}(\frac{\alpha u_t}{u}+\frac{u_{11t}}{u_{11}})+(-u_t)\sigma_2^{ij}[\frac{\alpha u_{ij}}{u}-\frac{\alpha u_i u_j}{u^2}
+\frac{u_{11ij}}{u_{11}}-\frac{u_{11i}u_{11j}}{u_{11}^2}+\delta_{ij}]\\
&=& \frac{\alpha}{u} + \frac{u_{11t}}{u_t u_{11}}+\frac{2\alpha}{u}+\frac{\alpha u_t\sigma^{ij}_2u_ju_j}{u^2}
-\frac{u_t \sigma^{ij}_2u_{11ij}}{u_{11}}+\frac{u_t\sigma^{ij}_2u_{11i}u_{11j}}{u_{11}^2}-u_t(n-1)\sigma_1\\
&=& \frac{3\alpha}{u}+\frac{u_t \sigma^{ij,mn}_2 u_{ij1}u_{mn1}}{u_{11}}+2\frac{u_{t1}^2}{u_{11}u_t^2}
+\frac{\alpha u_t\sigma^{ij}_2u_ju_j}{u^2}+\frac{u_t\sigma^{ij}_2u_{11i}u_{11j}}{u_{11}^2}-u_t(n-1)\sigma_1.
\end{eqnarray*}
Note that
\begin{equation*}
u_t\sigma^{ij,mn}_2u_{ij1}u_{mn1}=(-u_t)( \sum_{i\neq j}u_{ij1}^2-\sum_{i\neq j}u_{ii1}u_{jj1}).
\end{equation*}

Now we want to estimate the second term on the right side of the above equality.
Assume that $u_{11}\geq \sqrt{\frac{2(1-\epsilon)(n-2)}{(n-1) \epsilon m_2}}$ at $x_0$, here $\epsilon $ to be determined,
otherwise our Lemma holds true. Then, using Corollary \ref{co1}, we obtain
\begin{eqnarray*}
(-u_t)(-\sum_{i\neq j}u_{ii1}u_{jj1})&\geq&\frac{n-1}{2}\frac{[2\sigma_2u_{111}u_t^2-u_{11}u_{t1}]^2}{[(n-1)u_{11}^2u^2_t+2(n-2)\sigma_2u_t^2]}-\frac{u_{t1}^2}{2u_t^2}\\
&\geq& \frac{1-\epsilon}{2}\frac{[2\sigma_2u_{111}u_t^2-u_{11}u_{t1}]^2}{u_{11}^2u^2_t}-\frac{u_{t1}^2}{2u_t^2}.
\end{eqnarray*}
Using Cauchy inequality, we have
\begin{equation*}\label{061301}
[2\sigma_2u_{111}u_t^2-u_{11}u_{t1}]^2\geq (1-\frac{1-\epsilon}{4-\epsilon})(2\sigma_2u_{111}u_t^2)^2+(1-\frac{4-\epsilon}{1-\epsilon})(u_{11}u_{t1})^2.
\end{equation*}
Then
\begin{eqnarray*}
(-u_t)(-\sum_{i\neq j}u_{ii1}u_{jj1})&\geq&\frac{6(1-\epsilon)}{4-\epsilon}\frac{u_{111}^2}{u^2_{11}}-2\frac{u_{t1}^2}{u_{11}u_t^2}.
\end{eqnarray*}
By the inequality \eqref{co2-1},
\begin{eqnarray*}
\sigma_2\geq \frac{5}{7}\sigma^{11}_{2}u_{11},
\end{eqnarray*}
and set $\epsilon=\frac{1}{34}$, we get
\begin{eqnarray}\label{102901}
\frac{(-u_t)(-\sum_{i\neq j}u_{ii1}u_{jj1})}{u_{11}}\geq \frac{\frac{16}{15}(-u_t)\sigma^{11}_{2}u_{111}^2}{u_{11}^2}-\frac{3}{2}\frac{u_{t1}^2}{u_{11}u_t^2}.
\end{eqnarray}

Then,
\begin{eqnarray*}
L(P)(x_0,t_0)&\geq&\frac{3\alpha }{u}
+\frac{\alpha u_t\sigma_{2}^{ii}u_{i}^2}{u^2}+\frac{\sum_{i\neq 1}(2u_{11}-\sigma_{2}^{ii})(-u_t)u^{2}_{11i}}{u_{11}^2}+
\frac{\frac{1}{15}\sigma^{11}_{2}(-u_t)u_{111}^2}{u_{11}^2}
\\&&-u_t(n-1)\sigma_1.
\end{eqnarray*}
In view of \eqref{1-diff} and the Cauchy-Schwarz inequality, we have
\begin{eqnarray*}
-(\frac{u_i}{u})^2\geq -\frac{2}{\alpha^2}\frac{u_{11i}^2}{u_{11}^2}
-\frac{2}{\alpha^2}( x_i)^2.
\end{eqnarray*}
Thus,
\begin{eqnarray*}
L(P)(x_0,t_0)&\geq&\frac{3\alpha }{u}+(\frac{1}{15}-\frac{2}{\alpha})
\frac{(-u_t)\sigma^{11}_{2}u_{111}^2}{u_{11}^2}
+\frac{\sum_{i\neq 1}(2u_{11}
-(1+\frac{2}{\alpha})\sigma_{2}^{ii})(-u_t)u^{2}_{11i}}{u_{11}^2}
\\&&-\frac{2}{\alpha} \sigma_{2}^{ii}x_{i}^2
-(n-1)u_t\sigma_1.
\end{eqnarray*}
In view of \eqref{co2-2}, if we choose $u_{11}$ bigger than some constant
$C(n, \alpha, A)$ (otherwise our lemma holds true automatically), we have
\begin{eqnarray*}
2u_{11}-(1+\frac{2}{\alpha})\sigma_{2}^{ii}&\geq& (1-\frac{2}{\alpha})u_{11}
-(1+\frac{2}{\alpha})(n-2)\bigg((n-1)^2 \sqrt{\frac{2A}{n(n-1)u_{11}m_1}}+\frac{7(n-1)}{5u_{11}m_2}\bigg)
\\&\geq& (1-\frac{2}{\alpha})u_{11}
-\frac{C(n, \alpha, A,m_1)}{\sqrt{u_{11}}}-C(n, \alpha, A,m_2)>0.
\end{eqnarray*}
Next, if we choose $\alpha$ large, we obtain at $(x_0,t_0)$
\begin{eqnarray*}
0\geq L(P)(x_0,t_0)&\geq&\frac{3\alpha }{u}
-\frac{2}{\alpha} \sigma_{2}^{ii}x_{i}^2
-(n-1)u_t\sigma_1\\&\geq& \frac{3\alpha }{u}
-\frac{(n-1)u_t}{2}\sigma_1\\&\geq& \frac{3\alpha }{u}
-\frac{(n-1)u_t}{2}u_{11}\\
&\geq&\frac{3\alpha }{u}
+\frac{(n-1)m_2}{2}u_{11}.
\end{eqnarray*}
So, we finish the proof of our Lemma.
\end{proof}

\section{The proof of Theorem \ref{main}}

We now begin to prove Theorem \ref{main}.
\begin{proof}
The proof is standard \cite{LRW} \cite{Tr}.
Let $u$ be an entire solution of the equation \eqref{SQ}. For any constant $R>1$, we consider the set
$$\Omega_{R}=\{(y,t) \in \mathbb{R}^n\times(-\infty,0]: u(Ry,R^2 t)<R^2\}.$$
Let
$$v(y,t)=\frac{u(Ry,R^2t)-R^2}{R^2}.$$
We consider the following Dirichlet problem:
\begin{equation}\label{SQ-1}
\left\{
\begin{aligned}
&(-v_t)\sigma_2(D^2v)=1, \quad  in\ \  \Omega_R; \\
&v=0, \quad  on\ \  \partial_p\Omega_R.
\end{aligned}
\right.
\end{equation}
Since $$D^2_{y} v=D^2_{x} u,$$
clearly, $v$ is a 2-convex-monotone solution of \eqref{SQ-1} and satisfies $\sigma_{3}(D^2 v)\geq -A$.
Applying Lemma \ref{le3}, so we have the estimates,
\begin{eqnarray}\label{est-}
(-v)^{\alpha}|D^2 v|\leq C(n,A,m_1,m_2,diam(\Omega_R(t)),|v|_{C^0(\Omega_R)}).
\end{eqnarray}
Now using the quadratic growth condition in Theorem \ref{main} and monotone of the solution, we have
\begin{eqnarray*}
b|R y|^2-c\leq u(R y,0)\leq u(R y,R^2t)\leq R^2,
\end{eqnarray*}
which implies
\begin{eqnarray*}
|y|^2\leq \frac{1+c}{b}.
\end{eqnarray*}
 We now
consider the domain
$$\Omega^{\prime}_{R}=\{(y,t) \in \mathbb{R}^n\times(-\infty,0]: u(Ry,R^2t)<\frac{R^2}{2}\}\subset \Omega_R.$$
In $\Omega^{\prime}_{R}$, we have
$$-c-1\leq v(y)\leq-\frac{1}{2}.$$
Then \begin{eqnarray}\label{est-}
(-v)^{\alpha}|D^2 v|\leq C(n,A,m_1,m_2,b,c).
\end{eqnarray}
Hence, \eqref{est-} implies that
$$|D^2 v|\leq 2^{\alpha} C(n,A,m_1,m_2,b,c).$$
Note that,
$$D^2_{y} v=D^2_{x} u.$$
Thus, using the previous two formulas, we have
$$|D^2 u|\leq C(n,A,m_1,m_2,\alpha,b,c), \quad \mbox{in} \quad \widetilde{\Omega}_{R}=\{(x,t)\in \mathbb{R}^n\times(-\infty,0]: u(x,t)<\frac{R^2}{2}\},$$
where $C$ is a absolutely constant. The arbitrary of $R$ implies the
above inequality holds true in all over $\mathbb{R}^n\times(-\infty,0]$. Using Evans-Krylov theory (Theorem 4.2 in \cite{NT}), we have
$$|D^2 u|_{C^{\alpha}}\leq C(n, \alpha)\frac{|D^2 u|_{C^{0}}}{R^{\alpha}}\leq \frac{C(n, \alpha)}{R^{\alpha}}.$$
$$|u_t|_{C^{\alpha}}\leq \frac{C(n, \alpha)}{R^{\alpha}}.$$
Hence, we obtain our theorem by letting $R\rightarrow +\infty$.
\end{proof}

\end{document}